\documentclass[12pt,reqno]{amsart}

\usepackage{pgfplots}
\pgfplotsset{compat=newest}

\usepackage[utf8]{inputenc}
\usepackage[english]{babel}

\usepackage{geometry}
\usepackage{graphicx}
\usepackage{amssymb}
\usepackage{amsmath}
\usepackage{listings}
\usepackage{amscd}
\usepackage[mathscr]{eucal}
\usepackage{caption}
\usepackage[all]{xy}
\usepackage{tikz}
\usepackage{tikz-cd}
\usepackage{pgfplots,pgfcalendar}

\theoremstyle{plain}
\newtheorem{theorem}{Theorem} % one counter for all

\newtheorem{lemma}[theorem]{Lemma} % lemmas, corollaries, etc., are numbered consecutively
\newtheorem{corollary}[theorem]{Corollary}
\newtheorem{proposition}[theorem]{Proposition}

\theoremstyle{definition} % these environments are in roman style

\newtheorem{example}[theorem]{Example}

\newcommand{\adef}{\begin{defn}}
\newcommand{\zdef}{\end{defn}}
\newtheorem{defn}[theorem]{Definition}

\def\D{\mathbb D}

\newcommand{\KP}{{\sf K}\hspace{-1pt}{\sf P}}

\title{The Butterfly lemma}

\author{Jes\'{u}s M .F. Castillo}
\address{Institute of Mathematics Imuex\\ Universidad de Extremadura\\
Avenida de Elvas\\ 06071-Badajoz\\ Spain} \email{castillo@unex.es}

\author{Daniel Morales}
\address{Institute of Mathematics Imuex \\ Universidad de Extremadura\\
Avenida de Elvas\\ 06071-Badajoz\\ Spain} \email{danmorg@unex.es}

\thanks{The research has been supported by MICIN Project PID2019-103961GB-C21 and
project IB20038 de la Junta de Extremadura. The second author was supported by the grant BES-2017-079901 related to the
project MTM2016-76958-C2-1-P}

\begin{document}

\begin{abstract} The Butterfly lemma we present can be considered a reiteration theorem for differentials generated from a complex interpolation process for families of K\"othe spaces. The lemma will be used to clarify the effect of different configurations in the resulting differential (because although interpolation is an orientation-free process, the obtention of differentials is not) and to round off a few aspects of Kalton's interpolation theorem.\end{abstract}

\maketitle

\section{Introduction}

The action of this paper will take place in the category of K\"othe function spaces over a $\sigma$-finite measure space $(\Sigma,\mu)$ with their
$L_\infty$-module structure; a particular case of which are Banach spaces with an unconditional basis with their
associated $\ell_\infty$-structure. We denote by $L_0$ the space of all $\mu$-measurable
functions over $(\Sigma,\mu)$. A K\"othe space $X$ is an $L_\infty$-submodule of $L_0$ endowed with a norm such that if $f, g\in X$ and $|f|\leq |g|$ then $\|f\|_X\leq \|g\|_X$. A \emph{centralizer on $X$} is a homogeneous map $\Omega: X \to L_0(\mu)$ for which there is a constant $C$ such that, given $f\in L_\infty(\mu)$ and $x\in X$, $\Omega(fx)- f\Omega(x)\in X$ and
$\|\Omega(fx)- f\Omega(x)\|_X\leq C\|f\|_\infty  \|x\|_X$. An homogeneous map $F$ between Banach spaces is called \emph{bounded} if $\sup_{\|x\|\leq 1} \|Fx\| <\infty$. The centralizer $\Omega$ is said to be \emph{trivial} if there exists a linear map $L: X\to L_0(\mu)$ such that $\Omega - L: X\to X$ is bounded. Two centralizers $\Omega, \Psi$ will be said to be (boundedly) equivalent if their difference
$\Omega -\Psi$ is (bounded) trivial.

We refer to \cite[section 10]{kaltonmontgomery}, \cite{berlof} and \cite{Coifman1982} for all basic facts about, respectively, abstract complex interpolation, complex interpolation for couples and complex interpolation for families.
Regarding the aim of this paper, and reduced to its bare bones, complex interpolation requires an open domain $D$ of the complex field $\mathbb C$. The domain $D$ is usually the unit strip $\mathbb S = \{ z\in \mathbb C: 0< \mathrm{Re}(z)<1\}$ for interpolation of pairs and the unit disk $\mathbb D = \{ z\in \mathbb C: |z|<1\}$ for families. On the border $\partial D$ of this domain one sets a family $(X_\omega)_{\omega \in \partial D}$ of Banach spaces all of them continuously contained in some Banach space $\Sigma$. One forms then a certain Banach space $\mathcal F$ of holomorphic functions defined on $D$ and with values on $\Sigma$ for which one asks that the evaluation functionals $\delta_z: \mathcal F \to \Sigma$ are continuous. This forces the evaluation of the derivative functionals  $\delta_z': \mathcal F \to \Sigma$ to be also continuous. The interpolation space at $z\in D$ will be the space $X_z = \{f(z): f\in \mathcal F\}$ endowed with its natural quotient norm. The derived space at $z\in D$ is the space $dX_z = \{(f'(z), f(z)): f\in \mathcal{F}\}$ endowed with its natural quotient norm. The differential (or derivation) of the family at $z$ is the map $\Omega_z = \delta_z' B_z$, where $B_z: X_z\to \mathcal{F}$ is a homogeneous bounded selection for the evaluation map
$\delta_z:\mathcal{F}\to \Sigma$. Note that different choices of selection $B_z$ lead to boundedly equivalent differentials at $z$. If $\|B_z\|=K$ then $B_z(x)$ is called a $K$-extremal for $x$. Thus, the value of a differential at $x$ can be obtained fixing $K$ and calculating the derivative of $K$-extremals for $x$. Differentials of families of K\"othe spaces are centralizers.

Kalton \cite{kaltmem,kaltdiff} developed a deep theory connecting centralizers and complex interpolation families of K\"othe spaces.
A critical point for us is that while complex interpolation is orientation-free, in the sense that, for instance, $[X,Y]_\theta = [X,Y]_{1-\theta}$, differentiation is not: if the differential of $(X,Y)$ at $ \theta$ is $\Omega$, the differential of $(Y, X)$ at $ 1- \theta$ is $-\Omega$. An important example will illustrate this: the interpolation couple $(\ell_\infty, \ell_1)$ has $[\ell_\infty, \ell_1]_{1/2}=\ell_2$; and since a homogeneous bounded selection for the evaluation map is $B(x)(z)= x^{2z}$ then $\Omega_{1/2}(x) = 2x\log x$ for $\|x\|=1$. Instead, the interpolation couple $(\ell_1, \ell_\infty)$ still yields $[\ell_1, \ell_\infty]_{1/2}=\ell_2$, although a homogeneous bounded selection for the evaluation map is now $B(x)(z)= x^{2(1-z)}$ and thus the differential is $\Omega_{1/2}(x) = -2x\log x$ for $\|x\|=1$. Throughout this paper we will call \emph{Kalton-Peck map} (at $\theta$) to  the differential generated by the couple $(\ell_\infty, \ell_1)$ at $\theta$, namely $\KP_{\theta^{-1}} x = \theta^{-1} x \log x $ for $\|x\|=1$. The associated derived space is the celebrated Kalton-Peck $Z_{\theta^{-1}}$ space \cite{kaltpeck}.

\section{Factorization and centralizers} Let $X,Y$ be two K\"othe spaces on the same base space and let $0<\theta<1$. We can form the following quasi Banach spaces:
\begin{itemize}
\item $X^\theta= \{ f\in L_0: |f|^{1/\theta} \in X\}$ endowed with the norm $\|x\|_\theta = \| |x|^{1/\theta}\|^\theta$.
\item $XY = \{ f \in L_0: f=xy \; \mathrm{for \; some}\; x\in X, y\in Y\}$
endowed with the quasinorm $\|f\|_{XY} = \inf \{ \|x\|_X\|y\|_Y : f=xy,\; x\in X, y\in Y\}$. According to \cite[Lemma 2]{cabepoint}, if $X,Y$ are Banach spaces then $XY$ is (isomorphic to) a $1/2$-Banach space.
\item If $X,Y$ are Banach spaces, the Banach space $X^{1-\theta}Y^{\theta}$ is therefore formed by the functions $f \in L_0$ such that $f=ab$ for some $a\in X^{1-\theta}, b\in Y^{\theta}\}$, so that $|f| = |x|^{1-\theta} |y|^{\theta}$ for $x=a^{1/ 1-\theta}$ and $y=b^{1/\theta}$. The space is thus endowed with the norm $\|f\|=\inf \{\|x\|_{X}^{1-\theta}\|y\|_{Y}^{\theta}: x\in X, y\in Y,
|f|=|x|^{1-\theta}|y|^{\theta}\}.$
\end{itemize}

Shestakov proved in \cite{sesta} that $[X_0,X_1]_\theta = \overline{ X_0\cap X_1} \subset X_0^{1-\theta}X_1^\theta$ and that there exist pairs $(X_0, X_1)$ of K\"othe spaces for which $X_0 \cap X_1$ is not dense in $X_0^{1-\theta}X_1^\theta$ and therefore $[X_0,X_1]_\theta\neq X_0^{1-\theta}X_1^\theta$.
The pair $(X_0, X_1)$ will be called \emph{regular}, according to \cite{cwikel} if $X_0\cap X_1$ is dense in both $X_0$ and $X_1$. Kalton \cite[formula (3.2)]{kaltdiff} already observed that complex interpolation for regular couples and factorization are the same, thanks to Lozanovskii decomposition formula (see \cite[Theorem 4.6]{kaltonmontgomery}); precisely,
the complex interpolation space $[X_0, X_1]_\theta$ is isometric to $X_0^{1-\theta} X_1^\theta$. The differential can be obtained
from the factorization approach as follows: observe first that, by homogeneity, when calculating
$\|x\|=\inf \{\|x_0\|_{X_0}^{1-\theta}\|x_1\|_{X_1}^{\theta}\}$ as above one may always assume that $\|x_0\|_{X_0}=\|x_1\|_{X_1}$. When $\|x_0\|_{X_0}, \|x_1\|_{X_1} \leq K\|x\|$ we shall say that
$|x|=|x_0|^{1-\theta}|x_1|^{\theta}$ is a $K$-optimal decomposition for $x$. If we do not need to specify $K$ we will just say an \emph{almost optimal} factorization. We construct now an associated homogeneous map as follows: for $K$ fixed, determine a $K$-optimal decomposition $a_0(x),a_1(x)$ for $x\in X_0^{1-\theta} X_1^\theta$ and set
\begin{equation}\label{factpair}
\mho_\theta(x) = x\, \log \frac{|a_1(x)|}{|a_0(x)|}.
\end{equation}
Let us show that $\mho_\theta$ can be seen as the associated differential $\Omega_\theta$: just observe that since $\|x\| = \|a_0(x)\|_{X_0}=\|a_1(x)\|_{X_1}$, one can obtain a $K$-bounded selection for $[X_0, X_1]_\theta$ by setting $B_\theta(x)(z) = |a_0(x)|^{1-z} |a_1(x)|^z$, which in turn yields
\begin{equation}\label{kpgeneralized}
\Omega_\theta(x) = \delta_\theta'B_\theta(x) = |a_0(x)|^{1-\theta} |a_1(x)|^{\theta}
\log \frac{|a_1(x)|}{|a_0(x)|} =x\, \log \frac{|a_1(x)|}{|a_0(x)|}.
\end{equation}

Centralizers obtained via factorization are sensible to orientation too, even if $X^{1-\theta}Y^\theta = Y^{1-(1-\theta)}X^{1-\theta}$. The form
in which differentials can be obtained from factorization in the case of finite families $X(1), \dots, X(n)$ of K\"othe spaces distributed in $n$ arcs $A_1, \dots, A_n$ on the unit sphere $\mathbb T=\partial \mathbb D$  was obtained in \cite{ccfg}. An important role will be played by the following functions: given a connected arc $A$ of the unit circle $\mathbb T$, we will denote $\varphi_{A}$ an analytic function on $\mathbb{D}$ such that $\mathrm{Re}\, \varphi_A = \chi_{A}$ on $\mathbb T$. Given  $z_0\in \D$, $\mathrm{Re}\, \varphi_j(z_0) = \mu_{z_0}(A_j)$ and one has:

\begin{proposition} Let $\mu_{z_0}$ be the harmonic measure on $\mathbb T$ with respect to $z_0$. Then
\begin{equation*}\label{fact}X_{z_0} = X(1)^{\mu_{z_0}(A_1)} X(2)^{\mu_{z_0}(A_2)} \cdots X(n)^{\mu_{z_0}(A_n)}\end{equation*}
with associated differential
\begin{equation}\label{diff}\Omega_{z_0}(x) = x\left ( \varphi_1'(z_0)\log a_1(x) +  \cdots  + \varphi_n'(z_0)\log a_n(x)\right)
\end{equation} where $x = a_1(x)^{\mu_{z_0}(A_1)} \cdots a_n(x)^{\mu_{z_0}(A_n)}$ is an almost optimal factorization of $x$ in the decomposition above. \end{proposition}
\begin{proof} Fix $K$ and check that if $x = a_1(x)^{\mu_{z_0}(A_1)} \cdots a_n(x)^{\mu_{z_0}(A_n)}$ is a $K$-optimal factorization of $x$ then $F(z)= a_1(x)^{\varphi_1(z)} \cdots a_n(x)^{\varphi_n(z)}$ is a $K$-extremal
at $z_0$ for $x$. Finally, compute $F'(z_0)$. \end{proof}

 Let us show now how to perform the passage from a strip configuration to a disk configuration:  let $(X_0, X_1)$ be an interpolation pair with interpolation space $[X_0, X_1]_{1/2}$ and differential $\Omega$:

\begin{lemma}\label{configuration}$\;$
\begin{itemize}
\item Consider the arcs $A_+=\lbrace e^{i\theta}: 0<\theta<\pi\rbrace$ and $A_-=\lbrace e^{i\theta}: \pi<\theta<2\pi\rbrace$ and set $X_0$ on $A_+$ and $X_1$ on $A_-$. Complex interpolation for this family at $z=0$ yields the space $[X_0, X_1]_{1/2}$ with differential $\Omega$.
\item The passage from $\Omega$ to $i\Omega$ is obtained by a rotation of $\pi/2$ to the right of the original configuration of the family. Precisely: Consider the arcs $A_{left}=\lbrace e^{i\theta}: \pi/2<\theta<3\pi/2\rbrace$ and $A_{right}=\lbrace e^{i\theta}: 3\pi/2<\theta<\pi/2\rbrace$ and set $X_0$ on $A_{left}$ and $X_1$ on $A_{right}$. Complex interpolation for this family at $z=0$ yields the space $[X_0, X_1]_{1/2}$ with differential $i\Omega$.\end{itemize}
\end{lemma}

We connect now factorization and interpolation of finite families and sums of centralizers.

\begin{lemma}[Butterfly lemma]  Let $X_0,X_1$ and $Y_0, Y_1$ be two regular pairs of Köthe spaces on the same measure space. Fix $0<\eta, \theta<1$.
 Assume that $[X_0, Y_0]_\theta=[X_1, Y_1]_\theta$ and let $\Omega_0$ (resp. $\Omega_1$) be the associated derivations.
The derivation generated by the pair
$([X_0, X_1]_\eta, [Y_0, Y_1]_\eta) $ at $\theta$ is $(1-\eta) \Omega_0  + \eta \Omega_1.$\end{lemma}
\begin{proof} Observe that $$ \left [ [X_0, Y_0]_\theta, [X_1, Y_1]_\theta\right]_\eta  = X_0^{(1-\theta)(1-\eta)}X_1^{(1-\theta)\eta}Y_0^{\theta(1-\eta)}Y_1^{\theta\eta}=\left[[X_0, X_1]_\eta, [Y_0, Y_1]_\eta\right]_\theta.$$ An optimal factorization for $f\in X$ is
$f= f_{X_0}^{(1-\theta)(1-\eta)}f_{X_1}^{(1-\theta)\eta}f_{Y_0}^{(1-\eta)\theta}f_{Y_1}^{\eta\theta}.$ Since $[X_0, Y_0]_\theta=[X_1, Y_1]_\theta$, we can compute the centralizer $\Omega$ associated at $\left [ [X_0, X_1]_\eta, [Y_0, Y_1]_\eta\right]_\theta$: according to formula (\ref{factpair}) we have:

\begin{eqnarray*}
\Omega(f)&=&f\log \frac{ f_{Y_0}^{(1-\eta)} f_{Y_1}^{\eta}}{f_{X_0}^{(1-\eta)} f_{X_1}^{\eta}}\\
&=& f\left( \log \frac{ f_{Y_0}^{(1-\eta)}} {f_{X_0}^{(1-\eta)}} + \log \frac{f_{Y_1}^{\eta}} {f_{X_1}^{\eta}}
\right)\\
&=& f\left((1-\eta) \log \frac{ f_{Y_0}} {f_{X_0}} + \eta \log \frac{f_{Y_1}} {f_{X_1}}
\right)\\
&=& (1-\eta) \Omega_0(f) + \eta \Omega_1(f).\qedhere
\end{eqnarray*}
\end{proof}
The Butterfly lemma can be considered a reiteration-like result for differentials. Its name is due to the following diagram, that yields an idea of how $\Omega$ is obtained:
$$
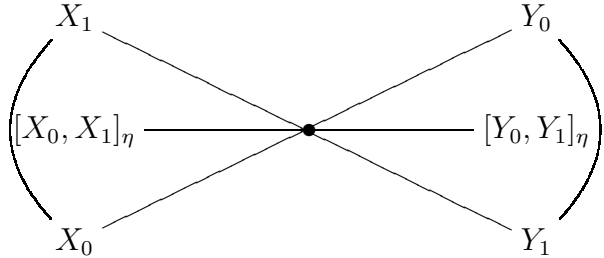
\xymatrix{
X_1 \ar@{-}[ddrrrr]&&&& Y_0\\
[X_0, X_1]_\eta\ar@{-}[rrrr]&&\bullet&& [Y_0, Y_1]_\eta\\
X_0\ar@{-}@/^2pc/[uu] \ar@{-}[uurrrr]&&&&Y_1\ar@{-}@/_2pc/[uu]}$$

\begin{corollary}\label{corbutter} Let $(X_0, X_1)$ and $(Y_0, Y_1)$ be two regular interpolation pairs such that $[X_0, X_1]_{1/2} = [Y_0, Y_1]_{1/2}$ with associated centralizers $\Omega_X$ and $\Omega_Y$. Then $\Omega_X + \Omega_Y$ is projectively equivalent to the centralizer generated at $1/2$ by the pair $([X_0, Y_0]_{1/2}, [X_1, Y_1]_{1/2})$.\end{corollary}

One can always reduce the situation to $z_0=0$ using a suitable conformal mapping. In that case, if $X(1), \dots,  X(n)$ are distributed in $n$ arcs of lengths $l_j$ then $X_{0} = X(1)^{l_1} \cdots X(n)^{l_n}$ with associated differential
$$\Omega_{0}(x) =  x\left ( \varphi_1'(0)\log a_1(x) +  \cdots  + \varphi_n'(0)\log a_n(x)\right)$$
(the functions $\varphi_i$ reflect the position of the arcs). An important case is when the spaces are distributed in $n$ arcs of equal length $A_j=\lbrace e^{i\theta}: \frac{2\pi(j-1)}{n}<\theta<\frac{2\pi j}{n}\rbrace$ for $j=1,\dots,n$ and $z_0=0$, in which case $X_{0} = X(1)^{1/n} \cdot ... \cdot X(n)^{1/n}$ and associated differential as above. In this case, calling $\omega_j = e^{\frac{2\pi j}{n}i}$, $j=1,\dots,n-1$, the $n^{th}$ roots of unity then $\varphi_1(z)=\varphi_2(\omega_1 z)=\varphi_3(\omega_2 z) = \cdots = \varphi_{n}(\omega_{n-1} z)$ and thus $\Omega_0$ can be written as
$$
\Omega_0(x)= \varphi_n'(0)x\left(\sum_{j=1}^{n-1} \omega_{n-j}\log a_j(x)\right).
$$

Let us explain now how addition of differentials is reflected in the configuration of the spaces. From now on all pairs will be assumed to be regular without further mention.

\begin{proposition}\label{isum} Consider now two interpolation pairs: $(X_0,X_1)$ and $(Y_0, Y_1)$ so that $[X_0,X_1]_{1/2}= [Y_0, Y_1]_{1/2}$, the former with differential $\Omega_X$ and the latter with differential $\Omega_Y$. Consider the four arcs $A_{0}=\lbrace e^{i\theta}: 0<\theta<\pi/2\rbrace$, $A_{\pi/2} =\lbrace e^{i\theta}: \pi/2<\theta<\pi\rbrace$, $A_{\pi} =\lbrace e^{i\theta}: \pi<\theta<3\pi/2\rbrace$, $A_{3\pi/2}=\lbrace e^{i\theta}: 3\pi/2<\theta<2\pi\rbrace$, of equal length. One has
\begin{enumerate}
\item The centralizer $\Omega_X + i \Omega_Y$ is projectively equivalent to the centralizer obtained from the configuration: $Y_0$ on $A_{0}$, $X_0$ on $A_{\pi/2}$, $Y_1$ on $A_{\pi}$ and $X_1$ on $A_{3\pi/2}$. Interchanging $Y_0$ and $Y_1$ one generates $\Omega_X - i \Omega_Y$, and interchanging $X_0$ and $X_1$ one generates $-\Omega_X + i \Omega_Y$. Finally, exchanging \emph{both} $Y_0$ and $Y_1$, \emph{and} $X_0$ and $X_1$ one generates $- \Omega_X - i \Omega_Y$. Interchanging $X_0,X_1$ by $Y_0,Y_1$ one generates $i\Omega_X + \Omega_Y$.
\item The centralizer $\Omega_X + \Omega_Y$ is projectively equivalent to the centralizer obtained from the following configuration: set $[X_0, Y_0]_{1/2}$ on the arc $A_+$ and $[X_1, Y_1]_{1/2}$ on the arc $A_-$. \end{enumerate}
\end{proposition}
\begin{figure}[ht]
\begin{center}
\begin{tikzpicture}
   \draw (0,0) circle (2cm);     % circunferencia
   \draw (0,-1.8) -- (0,-2.2);   % sur
   \draw (1.8,0) -- (2.2,0);     % este
   \draw (0,1.8) -- (0,2.2);     % norte
   \draw (-1.8,0) -- (-2.2,0);   % oeste
   \draw (-2,2) node {$X_0$};
   \draw (2,-2) node {$X_1$};
   \draw (2,2) node {$Y_0$};
   \draw (-2,-2) node {$Y_1$};
\end{tikzpicture}
\caption{This configuration generates a centralizer projectively equivalent to $\Omega_X + i \Omega_Y$.}
\end{center}
\end{figure}
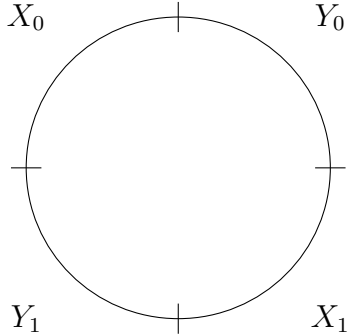

\begin{proof} To obtain (1) it is obviously enough to make the case $\Omega_X + i\Omega_Y$. Pick $x\in X_0^{1/2}X_1^{1/2}=Y_0^{1/2}Y_1^{1/2}$ and write  $x=x_0^{1/2}x_1^{1/2}=y_0^{1/2}y_1^{1/2}$. With the indicated distribution one obtains $X= X_0^{1/4}X_1^{1/4}Y_0^{1/4}Y_1^{1/4}$ so that if $x=x_0^{1/4}x_1^{1/4}y_0^{1/4}y_1^{1/4}$ is almost optimal then
\begin{eqnarray*}
\Omega(x)&=&\varphi_{A_{3\pi/2}}'(0)x\left(-i\log y_0-\log x_0+i\log y_1+\log x_1\right)\\
         &=&\varphi_{A_{3\pi/2}}'(0)x\left(\log \frac{x_1}{x_0}+i\log \frac{y_1}{y_0}\right)\\
         &=&\varphi_{A_{3\pi/2}}'(0)\left(\Omega_X+i\Omega_Y\right).\end{eqnarray*}
The proof of (2) follows from Lemma \ref{configuration} and Corollary \ref{corbutter}. \end{proof}

\section{Kalton's theorem revisited}

Kalton's obtained in \cite{kaltdiff} an outstanding theorem that establishes intimate relationships between centralizers and complex interpolation for families of K\"othe spaces. Recall that a centralizer $\Omega$ on a K\"othe space of functions $X$ is called \emph{real} if $\Omega(x)$ is real whenever $x\in X$ is real. We present a streamlined version according to \cite[Proposition 3.5]{cfg} where it is shown that some additional hypothesis Kalton includes are necessary only for existence, not for uniqueness:

\begin{theorem}\label{kalthm}\rm{[Kalton's theorem]}
\begin{itemize}
\item Given an interpolation pair $(X_0,X_1)$ of complex K\"othe spaces and $0<\theta<1$,
the derivation $\Omega_\theta$ is a real centralizer on $X_\theta$.
\item For every real centralizer $\Omega$ on a separable superreflexive K\"othe space $X$ satisfying some additional condition
there is a number $\varepsilon>0$ and an interpolation pair $(X_0, X_1)$ of K\"othe spaces
such that $X=X_\theta$ for some $0<\theta<1$ and $\varepsilon\Omega$ and $ \Omega_\theta$ are boundedly equivalent.
\item For any centralizer $\Omega$, real or not, satisfying the same additional conditions as before and defined on a separable superreflexive K\"othe space $X$ there is a number $\varepsilon>0$ and three K\"othe spaces $X(j)$, $j=1,2,3$ such that
if one considers the family $\mathcal X$ in which the three spaces are distributed in three arcs of equal length on the unit sphere
then $X=X_0$ and $\varepsilon\Omega$ and $\Omega_0$ are boundedly equivalent.
\end{itemize}
\end{theorem}

There are a few issues regarding Kalton's theorem that deserve consideration:
\begin{enumerate}
\item  The role of \emph{bounded equivalence}. In \cite{ccfg} it was shown that the differential $\Omega_\theta$ generated at $\theta$ by a pair of superreflexive K\"othe spaces is trivial if and only if there is a weight $w$ so that $X_0=X_1(w)$ up to
an equivalent renorming. Thus, up to a renorming, Kalton's theorem is valid for classical equivalence.
\item The strong stability for pairs of the theorem remains valid for families of up to three K\"othe spaces distributed in
three arcs of the unit circle but fails for families of four K\"othe
spaces \cite{ccfg}. The stability also fails \cite{qiu} for families of three spaces under arbitrary rearrangements.
\item Correa presented in \cite{correa} different examples of centralizers obtained from families of three K\"othe spaces that cannot be obtained from families of two K\"othe spaces. \end{enumerate}

Let re-examine (2) and (3). We begin showing how the Butterfly provides a clean proof \cite[Theorem 7.2 (i)]{kaltdiff}, which be instrumental in the proof of Kalton's theorem. Let $X$ be a K\"othe space. Lozanovskii's factorization yields $L_1=X\cdot X^*$. The indicator function of $X$ is defined in \cite[Section 5]{kaltdiff} as $\Phi_X(f)=f\log x$, whenever $f=x\cdot x^*$ is an almost optimal factorization. Moreover, given a centralizer $\Omega$ on $X$, its \emph{upload} to $L_1$ is defined in \cite[Theorema 5.1]{kaltmem} as $\Omega^{[1]}(f)=\Omega(x)x^*$. Theorem 7.2 establishes that if $X_0, X_1$ are K\"othe spaces and $X=[X_0, X_1]_\theta$ with induced centralizer $\Omega_\theta$ then for $f\in L_1$ almost optimally factorized as $|f|= x_0x_0^* = x_1 x_1^*$
$$\Omega^{[1]}(f)= f\left(\log x_1 - \log x_0\right)$$

This is an immediate consequence of the situation described in the diagram
$$
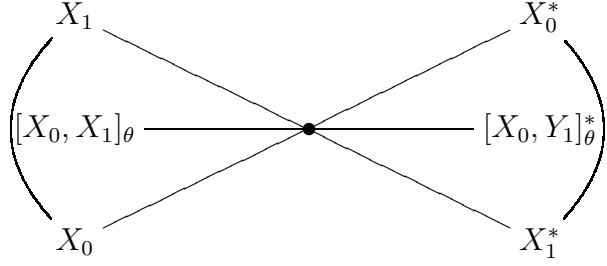
\xymatrix{
X_1 \ar@{-}[ddrrrr]&&&& X_0^*\\
[X_0, X_1]_\theta\ar@{-}[rrrr]&&\bullet&& [X_0, Y_1]_\theta^*\\
X_0\ar@{-}@/^2pc/[uu] \ar@{-}[uurrrr]&&&&X_1^*\ar@{-}@/_2pc/[uu]}$$

Indeed, set  $f=xx^*$ an almost optimal factorization of in $L_1=X X^*$ so that $f = xx^* = x_0^{1-\theta}x_1^\theta {x_0^*}^{1-\theta}{x_1^*}^\theta$ and then
\begin{eqnarray*}
\Omega^{[1]}(f) &=&  \Omega \left( x_0^{1-\theta}x_1^\theta\right)(x_0^*)^{1-\theta}(x_1^*)^\theta\\
&=& x \log \frac{x_1}{x_0} (x_0^*)^{1-\theta}(x_1^*)^{\theta}\\
&=& f\left(\log x_1 - \log x_0\right).\end{eqnarray*}

After that, the indicator of $\Omega$ is defined in \cite[Section 7]{kaltdiff} by $
\Phi^{\Omega}(f)=\int \Omega^{[1]}(f)\; d\mu.$ Assume now one has a family of four spaces $A, B, C, D$ equidistributed in the four arcs $A_{0}$, $A_{\pi/2}$, $A_{\pi}$ and $A_{3\pi/2}$ of Proposition \ref{isum}. The interpolated space at $z=0$ is $X=A^{1/4}B^{1/4}C^{1/4}D^{1/4}$, and it is then immediate that
\begin{eqnarray*}
\Phi^{\Omega}(f)&=&\int_0^{2\pi}e^{-it}\Phi_{X_{e^{it}}}(f)\; \frac{dt}{2\pi}\\
                &=&\int_0^{\pi/2}e^{-it}\Phi_A(f)\; \frac{dt}{2\pi}+\int_{\pi/2}^{\pi}e^{-it}\Phi_B(f)\; \frac{dt}{2\pi}+\\
                & &\qquad+\int_{\pi}^{3\pi/2}e^{-it}\Phi_C(f)\; \frac{dt}{2\pi}+\int_{3\pi/2}^{2\pi}e^{-it}\Phi_D(f)\; \frac{dt}{2\pi}\\
                &=&\dfrac{i}{2\pi}\big((-i-1)\Phi_A(f)+(-1+i)\Phi_B(f)+\\
                & &\qquad+(i+1)\Phi_C(f)+(1-i)\Phi_D(f)\big)\\
                &=&\dfrac{1+i}{2\pi}\big(-i\Phi_A(f)-1\Phi_B(f)+i\Phi_C(f)+\Phi_D(f)\big)
\end{eqnarray*}
(see also \cite[Prop.7.4]{kaltdiff}). Consider the pairs $(B,D)$ and $(A,C)$ with associated real centralizers at $0$ given by
$\Omega_B^D$ and $\Omega_A^C$. Then $\Phi^{\Omega_B^D} = \Phi_D - \Phi_B$ and $\Phi^{\Omega_A^C}= \Phi_C - \Phi_A$ and thus
$$\Phi^{\Omega}(f) = \dfrac{1+i}{2\pi}\big(\Phi^{\Omega_B^D}(f)+i\Phi^{\Omega_A^C}(f)\big)$$

This suggests the following reformulation of Kalton's theorem for complex interpolators:\\

$(\bigstar)$ Given a centralizer $\Omega$ on a K\"othe space $X$ there are four K\"othe spaces $X_0,X_1,Y_0,Y_1$ so that $\Omega$ is projectively equivalent to the centralizer generated at $0$ by the family formed by those spaces equidistributed in four arcs.\\

More precisely, let $ \Omega$ be a centralizer obtained at $0$ from a suitable interpolation family $(X_\omega)_{\omega\in \mathbb T}$ of K\"othe spaces. Then $\Omega$ is boundedly equivalent \cite{kaltdiff} to a centralizer having the form $\Omega_X + i\Omega_Y$ for two real centralizers $\Omega_X$ and $\Omega_Y$. By Kalton's theorem, there are interpolation pairs $(X_0, X_1)$ and $(Y_0, Y_1)$ such that $[X_0,X_1]_{1/2}= [Y_0,Y_1]_{1/2}$,
and such that  $\Omega_X$ is obtained from $(X_0, X_1)$ and $\Omega_Y$ is obtained from $(Y_0,Y_1)$. Proposition \ref{isum} yields the configuration that generates $\varphi_{A_{3\pi/2}}'(0)\left( \Omega_X + i\Omega_Y\right)$, which is projectively equivalent to $\Omega_X + i\Omega_Y$. Observe that if we rotate the arcs an angle $\alpha$, i.e. $\bar{A_j}=e^{\alpha i}A_j$, one obtain $\bar{\varphi_j}(z)=\varphi_j(e^{\alpha i}z)$ by the uniqueness of the functions $\varphi_j$, so its derivatives are $\bar{\varphi_j}'(z)=e^{\alpha i}{\varphi_j}'(e^{\alpha i}z)$ and the derivation is $\bar{\Omega}=e^{\alpha i}\Omega$ because of $\bar{\varphi_j}'(0)=e^{\alpha i}\varphi_j'(0)$. Thus, to obtain a clean equivalence one only has to rotate (counterclockwisely) in an angle $\alpha$ so that $e^{\alpha i}= |\varphi_{A_{3\pi/2}}'(0)| \varphi_{A_{3\pi/2}}'(0)^{-1}$ and replace the spaces $X_0, X_1$ (resp. ($Y_0, Y_1$)) by the spaces of their same scale $\bar{X_0}, \bar{X_1}$ (resp. ($\bar{Y_0}, \bar{Y_1}$)) that generate $|\varphi_{A_{3\pi/2}}'(0)|^{-1} \Omega_X$ (resp. $|\varphi_{A_{3\pi/2}}'(0)|^{-1} \Omega_Y$).\\

It is therefore superfluous to present a factorization theorem for more than four spaces, which explains why we focused on that case. Another issue that deserves attention is which centralizers cannot be obtained from families of two spaces. Correa \cite{correa} claims that ``in the search for such an example one feels tempted to use scales of $\ell_p$ spaces or $\ell_{p,q}$; however, reiteration results show that this approach is bound to fail." That is the reason why Correa constructs a sophisticated example using Orlicz spaces. Once again, using four spaces greatly simplifies the task. The ingredients we need are:
 \begin{itemize}
 \item \cite[Lemma 3.1]{correa}  If two nontrivial real centralizers $\Phi, \Psi$ are non-projectively equivalent then $\Phi + i\Psi$ is not projectively equivalent to a real centralizer .
 \item \cite[Example 1]{cabefac} The centralizer associated to the pair $(\ell_{p,2}, \ell_{p^*,2})$ is projectively equivalent to Kalton's centralizer $\mathcal K_2$ \cite{kaltcompo}.
 \item The centralizer associated to the pair $(\ell_{p}, \ell_{p^*})$ is projectively equivalent to Kalton-Peck's centralizer $\KP_2$ \cite{kaltmem,kaltdiff}.
\item \cite[Proposition 2]{cabefac} $\KP_2$  and $\mathcal {K}_2$ are not equivalent.
\end{itemize}

We are ready for the

\begin{example}\label{ejemplorentz} Consider the four arcs $A_0, A_{\pi/2}, A_{\pi}, A_{3\pi/2}$ from Proposition \ref{isum} and, for $1<p<\infty$,  set $X_\omega = \ell_{p,2}$ for $\omega\in A_0$, $X_\omega = \ell_p$ for $\omega\in A_{\pi/2}$, $X_\omega = \ell_{p^*,2}$ for $\omega\in A_{\pi}$ and $X_\omega = \ell_{p^*}$ for $\omega \in A_{3\pi/2}$. Interpolation space at $z=0$ yields the space
$$
\ell_2 = \ell_2^{\frac{1}{2}} \; \ell_{2}^{\frac{1}{2}}  = \ell_p^{\frac{1}{4}}\; \ell_{p^*}^{\frac{1}{4}} \; \ell_{p,2}^{\frac{1}{4}}\; \ell_{p^*,2}^{\frac{1}{4}}$$
with associated centralizer $\Omega$ projectively equivalent to $u\KP_2+ i v\mathcal{K}_2$ for two positive constants $u, v$. Since both $\KP_2$ and  $\mathcal{K}_2$ are nontrivial, and they are not projectively equivalent, $\Omega$ is not projectively equivalent to a real centralizer, hence it cannot be obtained from a family of two spaces.

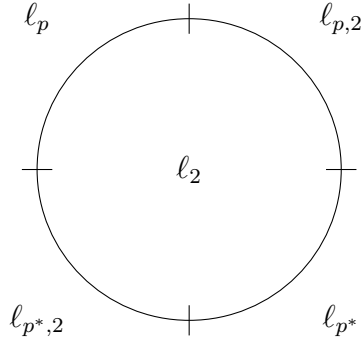
\begin{figure}[ht]
\begin{center}
\begin{tikzpicture}
   \draw (0,0) circle (2cm);     % circunferencia
   \draw (0,-1.8) -- (0,-2.2);   % sur
   \draw (1.8,0) -- (2.2,0);     % este
   \draw (0,1.8) -- (0,2.2);     % norte
   \draw (-1.8,0) -- (-2.2,0);   % oeste
   \draw (-2,2) node {$\ell_p$};
   \draw (2,-2) node {$\ell_{p^*}$};
   \draw (2,2) node {$\ell_{p,2}$};
   \draw (-2,-2) node {$\ell_{p^*,2}$};
   \draw (0,0) node {$\ell_2$};
\end{tikzpicture}
\caption{Configuration of Example \ref{ejemplorentz}.}
\end{center}
\end{figure}

\end{example}

\section{Twisted Hilbert spaces} We will consider now the most important case: when interpolated spaces are Hilbert spaces. Let $X$ be a K\"othe space such that $[X, X^*]_{1/2}=L_2$ with associated centralizer $\Omega_X$ and consider the scale $(L_\infty,L_1)$. The couple $[X, L_\infty]_\theta = X_\theta$ yields the $\theta^{-1}$-convexification of $X$. If one considers now the scale $(X_\theta,X_\theta^*)$ with associated centralizer $\Omega_{X_\theta}$ at $1/2$, we have
$$[X_\theta, X_\theta^*]_{1/2} = [[X,L_\infty]_\theta, [X^*, L_1]_\theta ]_{1/2} = [[X, X^*]_{1/2}, [L_\infty, L_1]_{1/2}]_\theta.$$

Since the centralizer associated to $[L_\infty, L_1]_{1/2}$ is $\KP_{2}$, one has:

\begin{proposition} Let $X$ be a K\"othe space such that $[X, X^*]_{1/2}=L_2$ with associated centralizer $\Omega_X$. The the centralizer generated at $[X_\theta, X_\theta^*]_{1/2}$ by the $\theta^{-1}$-convexification $X_\theta$ of $X$ is
\begin{equation} \label{desdif}
\Omega_{X_\theta} = (1-\theta) \Omega_X + \theta \KP_{2}.
\end{equation}
\end{proposition}

A clear counter-intuitive implication is that $\Omega_X$ and $\Omega_{X_\theta}$ cannot be simultaneously trivial. On the other extreme of triviality we encounter singularity: a quasilinear map $F$ is called singular if its restriction to any infinite dimensional subspace is not trivial, and corresponds to the property: the quotient map in the exact sequence $F$ generates is a strictly singular operator. Singular quasilinear maps behave differently from strictly singular operators: if $F,G$ are strictly singular operators then $F+G$ is strictly singular too. However, if $F,G$ are singular quasilinear maps, $F+G$ can be ($F+F=2F$) or not ($F-F=0$) singular. In \cite{newderivation} it is introduced the notion of \emph{strictly non-singular} quasilinear map $F$: one such that every infinite dimensional subspace contains a further infinite dimensional subspace on which $F$ is trivial. It is clear that if $F$ is singular and $G$ is strictly non-singular then $F+G$ is singular. One has

\begin{corollary} If $\Omega_{X}$ is strictly non-singular then, for all $0<\theta<1$, the map $\Omega_{X_\theta}$ is  singular.\end{corollary}

The Butterfly lemma also yields the identity $\KP_{2} = \frac{1}{\theta} \Omega_{X_{\theta}} - \frac{1-\theta}{\theta}\Omega_X$
from where it follows that, given two K\"othe spaces $X,Y$ as above, one has $\Omega_{X_{\theta}} - \Omega_{Y_{\theta}} = (1-\theta)\left( \Omega_X  - \Omega_Y\right)$. Therefore,

\begin{corollary} If $\Omega_{X}$ and $\Omega_Y$ are (boundedly) equivalent then $\Omega_{X_\theta}$ and $\Omega_{Y_\theta} $ are (boundedly) equivalent.\end{corollary}

and an extension/variation of \cite[Proposition 4.5]{ccfg}

\begin{corollary} $\Omega_X$ is trivial if and only if $X$ is a weighted $L_2(w)$ space.
\end{corollary}
\begin{proof} According to the formula (\ref{desdif}), if $ \Omega_X$ is trivial then $\Omega_{X_\theta}$ is equivalent to $\theta \mathcal K_{2}$. This yields, using Kalton's theorem above, that $X_\theta$ is $L_{2/\theta}$
and $X$ is, according to \cite[Proposition 4.2]{ccfg}, a weighted version of the Hilbert K\"othe space on the prefixed base space.\end{proof}

\end{document}